\documentclass[12pt,notcite,notref]{article}
\usepackage{amsmath}
\usepackage{amsthm}
\usepackage{epsfig,amssymb}
\usepackage{amsfonts}
\usepackage{amsopn}
\usepackage{array}
\usepackage{graphpap}
\usepackage{hyperref}\usepackage{latexsym}
\usepackage{epsfig,color,rotating}
\usepackage{color}%
\usepackage{subfig}

 \newtheorem{theorem}{Theorem}[section]
 \newtheorem{cor}[theorem]{Corollary}
 \newtheorem{lemma}[theorem]{Lemma}

 \newtheorem{remark}[theorem]{Remark}
 
\newtheorem{algo}[theorem]{Algorithm}

\usepackage[T1]{fontenc}
\usepackage[latin1]{inputenc}

\usepackage{bbm} 
\usepackage{amssymb}

\usepackage{algorithmic}
\usepackage[Algorithm, section]{algorithm}
\usepackage{MnSymbol}
\usepackage{grffile}
\usepackage{verbatim}
\usepackage[all]{xy}

\newcommand{\Hil}{\mathbbm{X}} 
\newcommand{\R}{\mathbbm{R}}

\newcommand{\N}{\mathbbm{N}}
\newcommand{\spr}[1]{\left\langle #1 \right\rangle} 

\newcommand{\bm}[1]{\mbox{\boldmath{$#1$}}} 
\newcommand{\bmsub}[1]{\mbox{\scriptsize \boldmath{$#1$}}}

\newcommand{\m}{\mathfrak m} 
\newcommand{\bfphi}{\bm{\varphi}}
\newcommand{\bfeta}{\bm{\eta}}
\newcommand{\bfp}{{\bm{p}}}
\newcommand{\bfu}{\boldsymbol{u}}
\newcommand{\bff}{{\bm{f}}}
\newcommand{\bfz}{{\bm{z}}}

\newcommand{\bfC}{{\bm{C}}}
\newcommand{\bfxi}{{\bm{\xi}}}

\newcommand{\bfv}{{\bm{v}}}

\newcommand{\Dual}{\mathbbm D}
\newcommand{\Phiad}{\Phi_{ad}}

\newcommand{\bfvarphi}{{\boldsymbol\varphi}}
\newcommand{\E}{\mathcal{E}} 
\newcommand{\bfg}{{\bm{g}}}
\newcommand{\Frechet}{Fréchet}
\newcommand{\Poincare}{Poincaré}
\newcommand{\relmiddle}{\mathrel{}\middle|\mathrel{}} 
\newcommand{\grossO}{\mathcal{O}} 

\def\smallint{\begingroup\textstyle \int\endgroup}

\begin{document}

\title{
An extension of the projected gradient method to a Banach space setting with application in structural topology optimization.}

\author{Luise Blank, Christoph Rupprecht}

\date{}
\maketitle

\begin{abstract}
For the minimization of a nonlinear cost functional $j$ under convex constraints the relaxed projected gradient process
\begin{align*}
	\varphi_{k+1} = \varphi_{k} + \alpha_k(P_H(\varphi_{k}-\lambda_k \nabla_H j(\varphi_{k}))-\varphi_{k})
\end{align*}
as formulated e.g. in \cite{DemyanovRubinov} is a well known method. The analysis is classically performed in a Hilbert space $H$. We generalize this method to functionals $j$ which are differentiable
in a Banach space. Thus it is possible to perform e.g. an $L^2$ gradient method if $j$ is only differentiable in $L^\infty$. We show global convergence using Armijo backtracking in $\alpha_k$ and allow the inner product and the scaling $\lambda_k$ to change
in every iteration. As application we present a structural topology optimization problem based on a phase field model, where the reduced cost functional $j$ is differentiable
in $H^1\cap L^\infty$. The presented numerical results using the $H^1$ inner product and a pointwise chosen metric including second
order information show the expected mesh independency in the iteration numbers. The latter yields an additional, drastic decrease in iteration numbers as well as in computation time. Moreover we present numerical results using a BFGS update of the $H^1$ inner product for further optimization problems based on phase field models.
\end{abstract}



\section{Introduction}
Let $j$ be a functional on a Hilbert space $H$ with inner product $(.,.)_H$ 
and induced norm $\|.\|_H$ and let $\Phiad\subseteq H$ be a non-empty,
convex and closed subset. We consider the optimization problem
\begin{align}\label{eq:optprobl}
	\min j(\varphi)\ \text{ subject to } \varphi\in \Phiad.
\end{align}
If $j$ is \Frechet{} differentiable with respect to $\|.\|_H$, 
the classical projected gradient method introduced in Hilbert space in
\cite{goldstein1964} and \cite{levitin1966constrained} can be applied, 
which moves in the direction of the negative $H$-gradient $-\nabla_H j\in H$, which is characterized by the equality
$(\nabla_Hj(\varphi),\eta)_H = \spr{j'(\varphi),\eta}_{H^*,H}$ $\forall\eta\in H$ and orthogonally projects the result 
back on $\Phiad$ to stay feasible, i.e.
\begin{align}\label{pro1}
	\varphi_{k+1} = P_H(\varphi_{k}-\lambda_k\nabla_H j(\varphi_{k})).
\end{align}
To obtain global convergence $\lambda_k$ has to be chosen according to some step length rule, which results in a gradient path method, or
one can perform a line search along the descent direction $v_k = P_H(\varphi_{k}-\lambda_k\nabla_H j(\varphi_{k})) - \varphi_k$.
A typical application is $H= L^2(\Omega)$, see e.g. \cite{KelleySachs92}.

In this paper we consider the case that $j$ is differentiable with respect to 
a norm which is not induced by a inner product. 
Hence no $H$-gradient $\nabla_H j$ exists. 
However, in Section \ref{sec:VMPT} we reformulate the method such that it is well defined 
under weaker conditions.
We show global convergence when Armijo backtracking is applied along $v_k$ and allow the inner product and the scaling $\lambda_k$ to change in every iteration.
We call this generalization `variable metric projection' type (VMPT) method.
In Section \ref{sec:Application} we study the applicability of the method to a structural topology optimization problem, namely
the mean compliance minimization in linear elasticity based on a phase field model. 
Then the reduced cost functional is differentiable only in  $H^1\cap L^\infty$.
In the last section we show numerical results for this mean compliance problem. As expected choosing the $H^1$ metric leads to
mesh independent iteration numbers in contrast to the $L^2$ metric. We also present the choice of a variable metric
using second order information and the choice of a BFGS update of the $H^1$ metric. 
This reduces the iteration numbers to less than a hundreth. 
Moreover, we give additional numerical examples for the successful application of the VMPT method.
These include a problem of compliant mechanism, drag minimization of the Stokes flow and an inverse problem.

\section{Variable metric projection type (VMPT) method}\label{sec:VMPT}
\subsection{Generalization of the projected gradient method}
The orthogonal projection $P_H(\varphi_{k}-\lambda_k\nabla_H j(\varphi_{k}))$ employed in \eqref{pro1} is the unique solution of
\begin{align*}
	\min_{y \in\Phiad}\frac{1}{2}\|(\varphi_{k}-\lambda_k\nabla_H j(\varphi_{k})) - y \|^2_{H},
\end{align*}
which is equivalent to the problem
\begin{align}\label{eq:proj2}
\min_{ y \in\Phiad}\frac{1}{2}\| y -\varphi_{k}\|^2_{H} 
+\lambda_k Dj(\varphi_{k}, y -\varphi_{k}),
\end{align}
since  
$(\nabla_H j(\varphi_{k}),y -\varphi_{k})_H = j'(\varphi_{k})(y -\varphi_{k})
=  Dj(\varphi_{k},y -\varphi_{k})$ where the last denotes the directional derivative of $j$ at
$\varphi_{k}$ in direction $  y -\varphi_{k}$.
If e.g. $ Dj(\varphi_{k}, y)$ is linear and continuous with respect to $y\in H$ the
cost functional of (\ref{eq:proj2}) is strictly convex, continuous and coercive in $H$, and hence  \eqref{eq:proj2}
has a unique solution $\bar\varphi_k$ \cite{Dacorogna}. 
In the formulation \eqref{eq:proj2} the existence of the gradient $\nabla_H j$ is not required. Even G\^ateaux differentiability can be omitted.\\
In the following we formulate an extension of the projected gradient method where
 $P_H(\varphi_{k}-\lambda_k\nabla_H j(\varphi_{k}))$ is replaced by the solution $\bar\varphi_k$ of \eqref{eq:proj2}.

First we drop the requirement of a gradient as mentioned above. We assume that the admissible set $\Phiad $ is a subset of an intersection of Banach spaces $\Hil\cap \Dual$, where $\Hil$ and $\Dual$ have certain properties (see \ref{ass:LimitsUnique}), which are e.g. fulfilled for $\Hil=H^1(\Omega)$ or $\Hil=L^2(\Omega)$ and $\Dual= L^\infty(\Omega)$. 
Furthermore assume that $j$ is continuously \Frechet{} differentiable on $\Phiad $ with respect to the norm 
$\|.\|_{\Hil\cap \Dual}:= \|.\|_{\Hil}+\|.\|_{\Dual}$.
The \Frechet{} derivative of $j$ at $\varphi$ is denoted by $j'(\varphi)\in (\Hil\cap \Dual)^*$ and we write
$\spr{.,.}$ for the dual paring in the space $\Hil\cap \Dual$. Moreover, we use $C$ as a positive universal constant throughout the paper.

Secondly, we also allow the 
norm $\|.\|_H$ in (\ref{eq:proj2}) to change in every iteration.
Therefore, we consider a sequence $\{a_k\}_{k\geq 0}$ of symmetric positive definite bilinear forms 
inducing norms $\|.\|_{a_k} $  on $\Hil\cap\Dual$ .
This approach falls into the class of variable metric methods and includes
the choice of Newton and Quasi-Newton based search directions 
(see for example \cite{Bertsekas, Dunn1980} and \cite{gruver1981algorithmic} for the unconstrained case). In \cite{Bertsekas} these methods are called scaled gradient projection methods and
in the case of $a_k = j''(\varphi_k)$ also constrained Newton's method. In finite dimension $a_k$ is given by $a_k(p,v):= p^T B_k v$ where $B_k$ can be
the Hessian at $\varphi_k$ or an approximation of it.

Hence, in each step of the VMPT method 
the projection type subproblem
\begin{align}
	\min_{y  \in\Phiad} \quad&\frac{1}{2}\| y -\varphi_k\| _{a_k}^2 
+ \lambda_k \spr{j'(\varphi_k),y -\varphi_k} 
\label{eq:projproblvm}
\end{align}
with some scaling parameter $\lambda_k >0$ has to be solved.
Problem \eqref{eq:projproblvm} is formally equivalent to the projection
$P_{a_k}(\varphi_{k}-\lambda_k\nabla_{a_k} j(\varphi_{k}))$. 
However, $j$ is not necessarily
differentiable with respect to $\|.\|_{a_k}$ and $\Hil\cap \Dual$ endowed with $a_k(.,.)$ is only a 
pre-Hilbert space. 
Hence $\nabla_{a_k} j(\varphi_{k})$ does not need to exist.
For globalization of the method we perform a line search 
based on the
widely used Armijo back tracking, which results in Algorithm \ref{algorithm1}.
In the next section it is shown that the algorithm is well defined under certain assumptions and 
in particular that a unique solution $\bar\varphi_k$ of \eqref{eq:projproblvm} exists, together with
the proof of convergence. We denote the solution of \eqref{eq:projproblvm} also by $\mathcal P_k(\varphi_k)$ due to the connection to
a projection.
\begin{algo}[VMPT method]\label{algorithm1}
\quad 
\begin{algorithmic}[1]
\STATE Choose $0<\beta < 1$, $0<\sigma<1$ and $\varphi_0 \in \Phiad$.
\STATE $k := 0$
\WHILE{$k \leq k_{\textrm{max}}$}
	\STATE Choose $\lambda_k$ and $a_k$.
	\STATE \label{choiceofoverlinevarphi} Calculate the minimum $\overline\varphi_k=\mathcal P_k(\varphi_k)$ of the subproblem \eqref{eq:projproblvm}.
	\STATE Set the search direction $v_k := \overline\varphi_k - \varphi_k$
	\IF{$\|\bfv_k\|_\Hil \leq \textrm{tol}$}
		\RETURN
	\ENDIF	
	\STATE Determine the step length $\alpha_k:= \beta^{m_k}$ with minimal $m_k\in\N_0$ such that \\
\label{eq:armijo}
  	{\begin{center} $j(\varphi_k + \alpha_k v_k) \leq j(\varphi_k) + \alpha_k \sigma \spr{j'(\varphi_k),v_k}$.\end{center}}
	\STATE Update $\varphi_{k+1} := \varphi_k + \alpha_kv_k$
	\STATE $k:=k+1$
\ENDWHILE
\end{algorithmic}
\end{algo}
The stopping criterion $\|v_k\|_\Hil \leq tol$ is motivated by the fact that $\varphi_k$ 
is a stationary point of $j$ if and only if $v_k=0$ and $v_k \rightarrow 0 $ in $\Hil$, 
cf. Corollary \ref{cor:vkzstat} and Theorem \ref{thm:GlobalConvvm}.\\

We would like to mention, that this algorithm is not a line search along the {\em gradient path}
, which is widely used (e.g. in \cite{Bertsekas,Dunn81,Dunn87,GawandeDunn,goldstein1964,gruver1981algorithmic,hinze2008optimization,KelleySachs92,Trol})
and which requires to solve a projection type subproblem like \eqref{pro1} in each line search iteration. This can be unwanted if calculating
the projection is expensive compared to the evaluation of $j$. To avoid this we perform a line search along the descent direction  $v_k$, which is suggested e.g. in finite dimension or in Hilbert spaces in \cite{Bertsekas,gruver1981algorithmic,rustem1984} and is also used in \cite{Dunn1980}.
To include the idea of the gradient path approach, we imbed the possibility to vary the scaling factor $\{\lambda_k\}_{k\geq 0}$ for the formal gradient in \eqref{eq:projproblvm} in each iteration.
The parameter $\lambda_k$ can be put into $a_k$ by dividing the cost in \eqref{eq:projproblvm} 
by $\lambda_k$. 
However, we treat it as a separate parameter since this reflects the case
where $a_k$ is fixed for all iterations.
Note that under the assumptions used in this paper a line search along the gradient path is not possible since not even the existence of a positive step length can be shown, cf. Remark \ref{rem:gradpathnotpossible}.\\

Moreover, there is a clear connection to sequential quadratic programming, considering that $\mathcal P_k(\varphi_k)$ is the solution of the quadratic approximation of $\min_{\varphi\in\Phiad} j(\varphi)$ with
\begin{align*}
	\min_{y\in\Phiad} j(\varphi_k) + \spr{j'(\varphi_k),y-\varphi_k} + \frac 1 2 a_k(y-\varphi_k,y-\varphi_k).
\end{align*}
However, the global convergence result is analysed by means of projected gradient theory.

\subsection{Global convergence result}
We perform the analysis of the method with respect to two norms in the spaces $\Hil$ and $\Dual$, which we assume to have the following properties:
\newcounter{AssCount}
\renewcommand{\theAssCount}{\textbf{(A{\arabic{AssCount}})}}    
\begin{list}{\theAssCount}{\usecounter{AssCount}}\setlength{\itemsep}{0pt}
 \item\label{ass:LimitsUnique} $\Hil$ is a reflexive Banach space. $\Dual$ is isometrically isomorphic to $\mathbbm B^*$, where $\mathbbm B$ is a separable Banach space. Moreover, for any sequence $\{\varphi_i\}$ in $\Hil\cap \Dual$ with $\varphi_i\to\varphi$ weakly in $\Hil$ 
and $\varphi_i\to\tilde\varphi$ weakly-* in $\Dual$, it holds $\varphi = \tilde \varphi$.
\end{list}
We identify $\Dual$ and $\mathbbm B^*$ and say that a sequence converges 
weakly-* in $\Dual$ if it converges weakly-* in $\mathbbm B^*$.
The separability of $\mathbbm B$ is used to get weak-* sequential compactness in $\Dual$.
We would like to mention that the results hold also if $\Dual $ is a reflexive Banach space, in particular if $\Dual$
is an Hilbert space. In this case weak-* convergence has to be replaced by weak convergence throughout the paper.
However, in the application we are interested in $\Dual = L^\infty (\Omega)$.\\

In case of the Sobolev space  $\Hil = W^{k,p}(\Omega)$ and $\Dual = L ^q(\Omega)$ 
where $\Omega\subseteq \R^d$ is a bounded domain, $k\geq 0$, $1< p< \infty$ and $1< q\leq \infty$
the above assumption is fulfilled.

In addition to the above conditions on $\Hil$ and $\Dual$ 
let the following assumptions hold for the problem (\ref{eq:optprobl}):
\begin{list}{\theAssCount}{\usecounter{AssCount}}\setlength{\itemsep}{0pt}\setcounter{AssCount}{1}
 \item\label{ass:Convex} \label{ass:Closed} $\Phiad\subseteq \Hil \cap \Dual $ is convex, closed in $\Hil$ and non-empty.
 \item\label{ass:Bdd} $\Phiad$ is  bounded in $\Dual$.
 \item\label{ass:BddBelow} $j(\varphi) \geq -C> -\infty$ for some $C>0$ and all $\varphi\in\Phiad$.
 \item\label{ass:Diff} $j$ is continuously 
differentiable in a neighbourhood of $\Phiad\subseteq\Hil\cap\Dual$. 
 \item\label{ass:WeakDiff} For each $\varphi\in \Phiad$ and for each sequence $\{\varphi_i\}\subseteq \Hil\cap \Dual$ 
with  $\varphi_i\to 0$ weakly in $\Hil$ and 
 weakly-* in $\Dual$ it holds $\spr{j'(\varphi),\varphi_i}\to 0$
 as $i\to\infty$.
\end{list}
Moreover, we request for the parameters $a_k$ and $\lambda_k$ of the algorithm that:
\begin{list}{\theAssCount}{\usecounter{AssCount}}\setlength{\itemsep}{0pt}\setcounter{AssCount}{6}
 \item \label{ass:aInnerProduct} $\{a_k\}$ is a sequence of symmetric positive definite bilinear forms on $\Hil\cap\Dual$.
 \item\label{ass:aCoercive} It exists $c_1 > 0$ such that $c_1\|p\|^2_{\Hil} \leq \|p\|^2_{a_k}$ for all $p\in\Hil\cap\Dual$ and $k\in\N_0$.
 \item\label{ass:aBdd} For all $k\in\N_0$ it exists $c_2(k)$ such that $\|p\|^2_{a_k}\leq c_2\|p\|^2_{\Hil\cap\Dual}$ for all $p\in\Hil\cap\Dual$.
 \item\label{ass:aWeakDiff} For all $k\in\N_0$, $p\in \Phiad$ and for each sequence $\{y_i\}\subseteq \Phiad$ 
where there exists some $y \in \Hil\cap \Dual $ with $y_i\to y$ weakly in $\Hil$ and weakly-* in $\Dual$ 
 it holds $a_k(p,y_i)\to a_k(p,y)$ as $i\to\infty$.
 \item\label{ass:aToZ}
For each subsequence $\{ \varphi_{k_i}\}_i$ of the iterates given by Algorithm \ref{algorithm1} converging in $\Hil\cap \Dual$,
the corresponding subsequence $\{ a_{k_i}\}_i$ has the property that $a_{k_i}(p_i,y_i)\to 0$
for any sequences $\{p_i\},\{y_i\}\subseteq \Hil\cap \Dual$ with $p_i\to 0$ strongly in $\Hil$ and weakly-* in $\Dual$ and $\{y_i\}$ converging in $\Hil\cap\Dual$.
 \item\label{ass:Lambda} It holds $0<\lambda_{min}\leq \lambda_k\leq \lambda_{max}$ for all $k\in\N_0$.
\end{list}
\ref{ass:LimitsUnique}-\ref{ass:Lambda} are assumed throughout this paper if not mentioned otherwise.\\
Assumption \ref{ass:aToZ} reflects the possibility of a point based choice of $a_k$, e.g. dependent on 
the Hessian $D^2 j(\varphi_{k})$ or on an approximation of the Hessian. Note that \ref{ass:aBdd}-\ref{ass:aToZ} is weaker than the assumption $\|p\|^2_{a_k}\leq c_2\|p\|^2_{\Hil}$. In \eqref{eq:akso} an example of $a_k$ is given, which only fulfills these weaker assumptions. Also \ref{ass:aCoercive} is weaker than $c_1\|u\|^2_{\Hil\cap\Dual} \leq \|u\|^2_{a_k}$. The main result of the paper is the following, which is proved in Section \ref{sec:Proof}.
\begin{theorem}\label{thm:GlobalConvvm}
Let $\{\varphi_k\}\subseteq\Phiad$ be the sequence generated by the VMPT method (Algorithm \ref{algorithm1}) with $tol=0$ and let the assumptions \ref{ass:LimitsUnique}-\ref{ass:Lambda} hold, then:
\begin{enumerate}
\item $\lim_{k\to\infty}j(\varphi_k)$ exists.
\item Every accumulation point of $\{\varphi_k\}$ in $\Hil\cap \Dual$ is a stationary point of $j$.
\item For all subsequences with $\varphi_{k_i}\to \varphi$ in $\Hil\cap \Dual$ where $\varphi$ is stationary, 
the subsequence $\{ v_{k_i}\}_i $ converges strongly in $\Hil$ to zero.
\item If additionally $j\in C^{1,\gamma}(\Phiad)$  with respect to  $\|.\|_{\Hil\cap \Dual} $ for some $0<\gamma\leq 1$ then 
the whole sequence $\{ v_k \}_k$ converges to zero in  $\Hil$.
\end{enumerate}
\end{theorem}
In the classical Hilbert space setting, i.e. $\Dual = \Hil = H$ for some Hilbert space $H$, the assumption \ref{ass:Bdd} can be dropped. Also assumption \ref{ass:WeakDiff} is trivial because of \ref{ass:Diff}. Moreover, assumptions \ref{ass:aInnerProduct}-\ref{ass:aToZ} are fulfilled for the choice $a_k(p,v) = (p,A_kv)_H$ where $A_k\in \mathcal L(H)$ is a self-adjoint linear operator with $m\|p\|_H^2\leq (p,A_kp)_h \leq M\|p\|_H^2$ and $M\geq m>0$ independent of $k$. This is e.g. assumed in the local convergence theory in \cite{Dunn87,GawandeDunn} and in finite dimension for global convergence in \cite{Bertsekas,rustem1984}. For the special choice $a_k(p,v) = (p,v)_H$, global convergence is shown in \cite{gruver1981algorithmic} and for the case of a line search along the gradient path in \cite{Dunn81}. Result 4. of Theorem \ref{thm:GlobalConvvm} is shown in \cite{hinze2008optimization} in case of a line search along the gradient path under the same assumption $j\in C^{1,\gamma}$. Thus the presented method is a generalization of the classical method in Hilbert space.\\

We would also like to mention the following:
\begin{remark}
If there exists $C>0$ such that $\|p\|_\Dual \leq C\|p\|_\Hil$ for all $p\in\Hil\cap\Dual$, assumption \ref{ass:Bdd} can be omitted.\\
If $\Hil$ is a Hilbert space, the choice $a_k(u,v) = (u,v)_H$ fulfills all assumptions \ref{ass:aInnerProduct}-\ref{ass:aToZ}.
\end{remark}

\subsection{Analysis and proof of the convergence result of the VMPT method}\label{sec:Proof}
We first show the existence and uniqueness of $\overline\varphi_k=\mathcal P_k(\varphi_k)$ based on the direct method in the calculus of variations using the following Lemma and
assumptions \ref{ass:Convex}, \ref{ass:Bdd} and \ref{ass:Diff}-\ref{ass:aWeakDiff}. 
Note that the standard proof cannot be applied, since $a_k$ is indeed $\Hil$-coercive, but 
$a_k$ and $\spr{j'(\varphi_k), \cdot }$ are not $\Hil$-continuous. 
Another difficulty is that $\Hil\cap\Dual$ is not necessarily reflexive.
\begin{lemma}\label{lem:wsconv}
 Let $\{p_k\}\subseteq\Phiad$ with $p_k\to p$ weakly in $\Hil$ for some $p\in\Phiad$. Then $p_{k}\to p$ weakly-* in $\Dual$.
\end{lemma}
\begin{proof}
Since $\Phiad $ is bounded in $\Dual$ and the closed unit ball of $\Dual$ is weakly-* sequentially compact due to the separability of $\mathbbm B$, 
we can extract from any subsequence of $\{p_k\} \subseteq \Phiad$ another subsequence $\{p_{k_i}\}$ with $p_{k_i} \rightarrow \tilde p$
weakly-* in $\Dual$ for some $\tilde p\in\Dual$. Due to the required unique limit in $\Hil$ and $\Dual$ we have $\tilde p =p$.
Since for any subsequence we find a subsequence converging to the same $p$, we have that the whole sequence converges to $p$.
\end{proof}
\begin{theorem} \label{thm:Projection}
For any $k\in\N_0$ and $\varphi \in\Phiad$, the problem 
\begin{align}
	\min_{y  \in\Phiad} \quad& \frac{1}{2}\| y -\varphi \| _{a_k}^2 
+ \lambda_k \spr{j'(\varphi),y -\varphi} 
\label{eq:projproblvm2}
\end{align}  admits a unique solution
$\bar\varphi := \mathcal P_k(\varphi)$, which is given by the unique solution of the variational inequality
\begin{align}\label{var}
a_k(\bar\varphi-\varphi,\eta- \bar\varphi) +\lambda_k \spr{j'(\varphi),\eta - \bar\varphi } \geq 0 \qquad \forall \eta \in \Phiad.
\end{align}
\end{theorem}
\begin{proof}
Let $k\in\N_0$ and $\varphi \in \Phiad$ arbitrary. Problem \eqref{eq:projproblvm2} is equivalent to 
\begin{align}\label{eq:projprobl2vm}
	\min_{y \in \Phiad } \quad& g_k( y ) := \tfrac{1}{2}a_{k}( y  , y ) + \spr{b_k, y }
\end{align}
where  $\spr{b_k, y } := \lambda_k \spr{j'(\varphi),y } - a_{k}(\varphi, y )$ and $b_k \in (\Hil\cap \Dual)^*$ due to \ref{ass:Diff} and \ref{ass:aBdd}.
By \ref{ass:Bdd} and \ref{ass:aCoercive} we get for any $y \in\Phiad$ with some generic $C>0$ 
\begin{align}
	g_k(y ) &\geq \frac{c_1}{2}\| y \|^2_{\Hil} - 
\|b_k\|_{(\Hil\cap \Dual)^*} (\| y \|_{\Hil}+\underbrace{\| y \|_{\Dual}}_{\leq C}) \geq -C. \label{eq:gbddbelowvm}
\end{align}
Thus $g_k$ is $\Hil$-coercive and bounded from below on $\Phiad$. Hence we can choose an infimizing sequence 
$\varphi_i\in \Phiad$, such that $g_k(\varphi_i)\xrightarrow{i\to\infty} \inf_{y \in\Phiad} g_k(y )$. 
From the estimate \eqref{eq:gbddbelowvm} we conclude that $\{\varphi_i\}_i$ is bounded in $\Hil$. Therefore, 
we can extract a subsequence (still denoted by $\varphi_i$) which converges weakly in $\Hil$ to some 
$\bar \varphi\in \Hil$. Since $\Phiad$ is convex and closed in $\Hil$, 
it is also weakly closed in $\Hil$ and thus $\bar \varphi \in \Phiad$. 
By Lemma \ref{lem:wsconv} we also get $\varphi_i\to \bar \varphi$ weakly-* in $\Dual$. 
Finally we show $g_k(\bar \varphi)=\inf_{y\in \Phiad} g_k(y)$. 
Using \ref{ass:WeakDiff}, \ref{ass:aCoercive} and \ref{ass:aWeakDiff} 
one can show that $\liminf_i a_k(\varphi_i,\varphi_i) \geq a_k(\bar \varphi,\bar \varphi)$ 
and $\lim_i \spr{b_k,\varphi_i}=\spr{b_k,\bar \varphi}$, thus $\liminf_i g_k(\varphi_i) \geq g_k(\bar \varphi)$. 
We conclude
\begin{align*}
\inf_{y \in \Phiad} g_k(y ) \leq g_k(\bar \varphi) \leq \liminf_i g_k(\varphi_i) = \inf_{y \in \Phiad} g_k(y),
\end{align*}
which shows the existence of a minimizer of \eqref{eq:projprobl2vm}. Using \ref{ass:aCoercive}, 
the uniqueness follows from strict convexity of $g_k$.\\
Due to \ref{ass:Diff} and \ref{ass:aBdd}, we have that $g_k$ is differentiable in $\Hil\cap\Dual$, where its directional derivative at $\bar \varphi$ in direction $\eta-\bar \varphi$ for arbitrary $\eta\in\Phiad$ is given by
\begin{align*}
	\spr{g_k'(\bar \varphi),\eta-\bar \varphi} &= a_{k}(\bar \varphi-\varphi,\eta-\bar \varphi) + 
\lambda_k\spr{j'(\varphi),\eta-\bar \varphi} \;.
\end{align*}
Since the problem \eqref{eq:projproblvm2} is convex, it is equivalent to the first order optimality condition, which is given by the variational inequality \eqref{var}, see
\cite{Trol}.
\end{proof}
We see that $\varphi\in\Phiad$ is a stationary point of $j$, i.e. $\spr{j'(\varphi),\eta-\varphi}\geq 0$ $\forall\eta\in\Phiad$, if and only if
$\overline\varphi=\varphi$ is the solution of \eqref{var}, i.e. the fixed point equation $\varphi=\mathcal P_k(\varphi)$ is fulfilled.
This leads to the classical view of the method as a fixed point iteration $\varphi_{k+1}=\mathcal P_k(\varphi_k)$ in the case that $\mathcal P_k$ is independent of $k$ and $\alpha_k=1$ is chosen.
\begin{cor}\label{cor:vkzstat}
If there exists some $k\in \N_0$ with $\mathcal P_k(\varphi) = \varphi$ then
$\varphi$ is a stationary point of $j$.
On the other hand, if $\varphi\in\Phiad$ is a stationary point of $j$ then 
the fix point equation $\mathcal P_k(\varphi) = \varphi$ holds
for all $k\in\N_0$.
In particular, an iterate $\varphi_k$ of the algorithm is a stationary point of $j$ if and only if 
$v_k= \mathcal P_k(\varphi_k)-\varphi_k = 0$. 
\end{cor}

The variational inequality (\ref{var}) tested with $\eta= \varphi\in \Phiad$ together with \ref{ass:aCoercive} and \ref{ass:Lambda} yields that 
$\mathcal P_k(\varphi)-\varphi$ is a descent direction for $j$:
\begin{lemma}\label{lem:descdirvm}
	Let $k\in\N_0$, $\varphi\in\Phiad$ and $v:=\mathcal P_k(\varphi)-\varphi$. Then it holds
	\begin{align}\label{eq:descdirvm}
		\spr{j'(\varphi),v}\ \leq -\frac {c_1} {\lambda_{max}} \|v\|^2_{\Hil}.
	\end{align}\qed
\end{lemma}
Note that \eqref{eq:descdirvm} does not hold in the $\Hil\cap\Dual$-norm.\\
Due to $\spr{j'(\varphi),v}<0$ for $v\neq 0$ the step length selection by the Armijo rule (see step \ref{eq:armijo} in Algorithm \ref{algorithm1}) is well defined, which can be shown as in \cite{Bertsekas}.
\begin{remark}\label{rem:gradpathnotpossible}
For the existence of a step length and for the global convergence proof we exploit that the path $\alpha\mapsto \varphi_k + \alpha v_k$ is continuous in $\Hil\cap \Dual$. Thus, also the mapping $\alpha\mapsto j(\varphi_k + \alpha v_k)$ is continuous. On the other hand, this does not hold for the gradient path. Backtracking along the gradient path or projection arc means that $\alpha_k$ is set to 1, whereas $\lambda_k=\beta^{m_k}$ is chosen with $m_k\in\N_0$ minimal such that the Armijo condition
\begin{align*}
	j(\overline\varphi_k(\lambda_k))\leq j(\varphi_k) + \sigma \spr{j'(\varphi_k),\overline\varphi_k(\lambda_k)-\varphi_k}
\end{align*}
is satisfied, see for instance \cite{KelleySachs92}. By the notation $\overline\varphi_k(\lambda_k)$ we emphasize that the solution of the subproblem \eqref{eq:projproblvm} depends on $\lambda_k$. However, with the above assumptions it cannot be shown that there exists such a $\lambda_k$. The reason is that due to \ref{ass:aCoercive} the gradient path $\lambda\mapsto \overline\varphi_k(\lambda)$ is continuous with respect to the $\Hil$-norm, whereas $j$ is due to \ref{ass:Diff} only differentiable with respect to the $\Hil\cap\Dual$-norm. Thus, $j$ along the gradient path, i.e. the mapping $\lambda\mapsto j(\overline\varphi_k(\lambda))$, may be discontinuous.
\end{remark}
To prove statement 2. of Theorem \ref{thm:GlobalConvvm} we use, as in \cite{Bertsekas} for finite dimensions, that $v_k$ is gradient related. This is weaker than
the common angle condition. Therefor we need the following
two lemmata:
\begin{lemma}\label{lem:Tech}
For $\{\varphi_k\}_k \subseteq\Phiad$ with $\varphi_k\to\varphi$ in $\Hil\cap\Dual$ and $\{p_k\}_k \subseteq \Hil\cap\Dual$ with  $p_k\to p$ 
weakly in $\Hil$ and weakly-* in $\Dual$ for some $\varphi,p\in\Hil\cap\Dual$ it holds
$\spr{j'(\varphi_k),p_k}\to \spr{j'(\varphi),p}$.
\end{lemma}
\begin{proof}
We use \ref{ass:Diff} and \ref{ass:WeakDiff} and obtain
\begin{align*}
	&|\spr{j'(\varphi_k),p_k}- \spr{j'(\varphi),p}|\leq |\spr{j'(\varphi_k)-j'(\varphi),p_k}| + | \spr{j'(\varphi),p_k-p}| \leq\\
	&\leq \underbrace{\|j'(\varphi_k)-j'(\varphi)\|_{(\Hil\cap \Dual)^*}}_{\to 0}\underbrace{\|p_k\|_{\Hil\cap \Dual}}_{\leq C} + \underbrace{| \spr{j'(\varphi),p_k-p}|}_{\to 0}\to 0. \qedhere
\end{align*}
\end{proof}
The preceding lemma is also needed in the proof of Theorem \ref{thm:GlobalConvvm}.

\begin{lemma}\label{thm:PHBdd}
Let for a sequence $\{\varphi_i\}_i\subseteq \Phiad$ hold $\varphi_i\to \varphi$ in $\Hil\cap \Dual$ for some $\varphi\in\Hil\cap\Dual$. 
Then there exists $C>0$ such that $\|\mathcal P_{k}(\varphi_i)\|_{\Hil\cap\Dual}\leq C$ for all $i,k\in\N_0$.
\end{lemma}
\begin{proof}
Lemma \ref{lem:descdirvm} yields together with \ref{ass:Bdd} and \ref{ass:Diff} the estimate
\begin{align*}
	\tfrac{c_1}{\lambda_{max}}\|\mathcal P_{k}(\varphi_i)-\varphi_i\|^2_\Hil&\leq -\spr{j'(\varphi_i),\mathcal P_{k}(\varphi_i)-\varphi_i}\\
	&\leq \|j'(\varphi_i)\|_{(\Hil\cap\Dual)^*}(\|\mathcal P_{k}(\varphi_i)-\varphi_i\|_\Hil + \|\mathcal P_{k}(\varphi_i)-\varphi_i\|_\Dual)\\
	&\leq C(\|\mathcal P_{k}(\varphi_i)-\varphi_i\|_\Hil+1),
\end{align*}
thus $\|\mathcal P_{k}(\varphi_i)-\varphi_i\|_\Hil \leq C$ and hence $\|\mathcal P_{k}(\varphi_i)\|_\Hil\leq C$. Due to \ref{ass:Bdd} we finally get $\|\mathcal P_{k}(\varphi_i)\|_{\Hil\cap\Dual}\leq C$ independent of $i$ and $k$.
\end{proof}

\begin{lemma}\label{thm:gradrelated}
Let $\{\varphi_k\}$ be the sequence generated by Algorithm \ref{algorithm1}, then $\{ v_k\}_k$ is gradient related, i.e.:
for any subsequence $\{ \varphi_{k_i} \}_i$ which converges in $\Hil\cap \Dual$ to a nonstationary point
$\varphi \in \Phiad$ of $j$, the corresponding subsequence of search directions $\{ v_{k_i}\}_i$ 
is bounded in $\Hil\cap \Dual$ and $\limsup_i \spr{j'(\varphi_{k_i}),v_{k_i}} < 0$ is satisfied.
Moreover, it holds $\liminf_i \|v_{k_i}\|_{\Hil}> 0$. 
\end{lemma}
\begin{proof}
Let $\varphi_{k_i}\to \varphi$ in $\Hil\cap\Dual$, where $\varphi$ is nonstationary.
Lemma \ref{thm:PHBdd} provides that $\{v_{k_i}\}_i $ 
is bounded in $\Hil\cap\Dual$. 
With \eqref{eq:descdirvm}, the statement $\limsup_i \spr{j'(\varphi_{k_i}),v_{k_i}} < 0$  follows from
$\liminf_i \|v_{k_i}\|_{\Hil}= C> 0$, which we show by contradiction.\\
Assume $\liminf_i \|v_{k_i}\|_{\Hil} = 0$, thus there is a subsequence again denoted by $\{v_{k_i}\}_i $ such that 
$v_{k_i}\to 0$ in $\Hil$. 
Using \eqref{var} for $\bar\varphi_k:= \mathcal P_{k}(\varphi_k)$, the positive definiteness of $a_k$ and \ref{ass:Lambda}, it follows for all $\eta \in \Phiad $
\begin{align}
\spr{j'(\varphi_{k}),\eta-\bar\varphi_{k}} &\geq 
\tfrac{1}{\lambda_k} (a_k(v_k,v_k) +  a_k(v_k,\bar\varphi_k-v_k -\eta))\nonumber \\
&\geq  -\tfrac{1}{\lambda_{min}} |a_k(v_k,\bar\varphi_k-v_k -\eta)| \; . \label{eq:fstorderki}
\end{align}
Moreover, $\bar \varphi_{k_i} = v_{k_i}+\varphi_{k_i}\to \varphi$ in $\Hil$ and 
also weakly-* in $\Dual$ according to Lemma \ref{lem:wsconv}. 
From Lemma \ref{lem:Tech} we get $\spr{j'(\varphi_{k_i}),\eta-\bar\varphi_{k_i}}\to \spr{j'(\varphi),\eta-\varphi}$. 
From \ref{ass:aToZ} we get $a_{{k_i}}(\bar \varphi_{k_i}-\varphi_{k_i}, \varphi_{k_i}-\eta)\to 0$ and we derive from \eqref{eq:fstorderki} that
\begin{align*}
	\spr{j'(\varphi),\eta-\varphi} \geq 0\quad \forall \eta\in\Phiad,
\end{align*}
which shows that $\varphi$ is stationary, which is a contradiction.
\end{proof}
\label{pageofproof}\begin{proof}[Proof of Theorem \ref{thm:GlobalConvvm}]\ \\
Because of Corollary \ref{cor:vkzstat} we can assume $v_k\neq 0$ and $\alpha_k>0$ for all $k$.\\
1.) From the Armijo rule and since $v_k$ is a descent direction we get 
\begin{align}\label{eq:armijo2}
 	j(\varphi_{k+1})-j(\varphi_k) \leq  \alpha_k \sigma \spr{j'(\varphi_k),v_k} < 0,
\end{align}
thus $j(\varphi_k)$ is monotonically decreasing. Since $j$ is bounded from below
we get convergence $j(\varphi_k)\to j^*$ for some $j^*\in\R$, which proves 1. \\[2mm]
2.)
The proof is similar to \cite{Bertsekas} in finite dimension by contradiction. Let $\varphi$ be an accumulation point, with a convergent subsequence $\varphi_{k_i}\to \varphi$ in $\Hil\cap\Dual$.
The continuity of $j$ on $\Phiad$ yields then $j^* = j(\varphi)$ and \eqref{eq:armijo2}
leads to $\alpha_k \spr{j'(\varphi_k),v_k}\to 0$.
Assuming now that $\varphi$ is nonstationary we have $\left|\spr{j'(\varphi_{k_i}),v_{k_i}}\right| \geq C > 0$, since
$\{v_k\}$ is gradient related by Lemma \ref{thm:gradrelated}, 
and thus $\alpha_{k_i}\to 0$.
So there exists some $\bar i\in\N$ such that $\alpha_{k_i}/\beta\leq 1$ for all $i\geq \bar i$, 
and thus $\alpha_{k_i}/\beta$ does not fulfill the Armijo rule due to the minimality of $m_k$.
Applying the mean value theorem to the left hand side, we have for some nonnegative $\tilde\alpha_{k_i} \leq \frac{\alpha_{k_i}}{\beta} $
and all $ i\geq\bar i$ that
\begin{align}
 	\tfrac{\alpha_{k_i}}{\beta}\spr{j'\left(\varphi_{k_i} + \tilde\alpha_{k_i} v_{k_i}\right),v_{k_i}} = 
	j\left(\varphi_{k_i} + \tfrac{\alpha_{k_i}}{\beta} v_{k_i}\right)-j(\varphi_{k_i}) >  
\tfrac{\alpha_{k_i}}{\beta} \sigma \spr{j'(\varphi_{k_i}),v_{k_i}}
\label{eq:ineq1}
\end{align}
holds. Since, by Lemma \ref{thm:gradrelated}, $\{v_{k_i}\}_i$ is bounded in $\Hil\cap \Dual$ and $\tilde\alpha_{k_i}\to 0$, 
we have that $\varphi_{k_i} + \tilde\alpha_{k_i} v_{k_i}\to \varphi$ in $\Hil\cap \Dual$.
Also $\bar\varphi_{k_i} = \varphi_{k_i} + v_{k_i}$ is uniformly bounded in $\Hil\cap\Dual$ and thus there exists a subsequence, again denoted by $\{\bar\varphi_{k_i}\}$,
which converges to some $y\in\Phiad$ weakly in $\Hil$ and  weakly-* in $\Dual$.
Hence we have that $v_{k_i} = \bar\varphi_{k_i}-\varphi_{k_i}\to \bar v := y-\varphi$ weakly in $\Hil$ and weakly-* in $\Dual$.
According to Lemma \ref{lem:Tech} we can take the limit of both sides of the inequality \eqref{eq:ineq1},
which leads to 
$
 	\spr{j'\left(\varphi\right),\bar v} \geq \sigma \spr{j'\left(\varphi\right),\bar v},
$
and $\sigma < 1$ yields
 $	\spr{j'\left(\varphi\right),\bar v} \geq 0 \; .$
This contradicts 
 $	\spr{j'\left(\varphi\right),\bar v} = \limsup_i\spr{j'(\varphi_{k_i}),v_{k_i}} < 0,$
which is a consequence of Lemma \ref{thm:gradrelated}.\\[2mm]
3.)
By proving that out of 
any subsequence of $\spr{j'(\varphi_{k_i}),v_{k_i}}$ we can extract another subsequence, which converges to 0, we can conlude
that $\spr{j'(\varphi_{k_i}),v_{k_i}}\to 0$ which yields $\| v_{k_i} \|_\Hil \to 0$ by \eqref{eq:descdirvm}.
With Lemma \ref{thm:PHBdd}, we get by the same arguments as in 2. that $v_{k_i} \to y-\varphi$ weakly in $\Hil$ and weakly-* in $\Dual$ for a subsequence and for some $y\in \Phiad$, thus $\spr{j'(\varphi_{k_i}),v_{k_i}}\to \spr{j'(\varphi),y-\varphi}$ due to Lemma \ref{lem:Tech}.
Since $v_{k_i}$ are descent directions for $j$ at $\varphi_{k_i}$ and $\varphi$ is stationary we have
$\spr{j'(\varphi),y-\varphi}=0$.
\\[2mm]
4.)
As in 3.) we prove by a subsequence argument that $\spr{j'(\varphi_k),v_k}\to 0$.
For an arbitrary subsequence, which we also denote by index $k$, \eqref{eq:armijo2} yields 
$\alpha_k\spr{j'(\varphi_k),v_k}\to 0 $. If  $\alpha_k\geq c> 0$ for all $k$, the assertion follows immediately. 
Otherwise there exists a subsequence (again denoted by index $k$) such that $\beta \geq \alpha_k\to 0$ 
and thus the step length $\alpha_k/\beta$ does not fulfill the Armijo condition.
Since $j'$ is Hölder continuous with exponent $\gamma$ and modulus $L$ 
we obtain
\begin{align*}
\sigma\tfrac{\alpha_k}{\beta}\spr{j'(\varphi_k),v_k} &< j(\varphi_k+\tfrac{\alpha_k}{\beta}v_k)-j(\varphi_k) 
= \int_0^1 \tfrac{d}{dt} j(\varphi_k+t \tfrac{\alpha_k}{\beta}v_k) dt \nonumber \\
&\leq \tfrac{\alpha_k}{\beta}\spr{j'(\varphi_k),v_k}+\tfrac{L}{1+\gamma}  
\left(\tfrac{\alpha_k}{\beta}\right)^{1+\gamma}\|v_k\|^{1+\gamma}_{\Hil\cap\Dual}.
\end{align*}
It holds $\|v_k\|_{\Dual}\leq C$ due to \ref{ass:Bdd} and employing \eqref{eq:descdirvm}
we obtain
\begin{align*}
0<(\sigma-1)\spr{j'(\varphi_k),v_k}< C\tfrac{L}{1+\gamma} 
(\tfrac{\alpha_k}{\beta})^\gamma(\|v_k\|^{1+\gamma}_{\Hil}+ 1 )
\leq 
C \alpha_k^\gamma ( |\spr{j'(\varphi_k),v_k}|^{\frac{1+\gamma}{2}}+ 1 ).
\end{align*}
We get
$x_k:= | \spr{j'(\varphi_k),v_k}|\to 0$. Otherwise there exists a subsequence still denoted by $\{ x_k  \} $ 
with $x_k \to \bar c>0 $. 
Rearranging the last inequality gives
$1 < C\alpha_k^\gamma ( x_k^{\frac{-1+\gamma}{2}}+ x_k^{-1} )\to 0$,
	which is a contradiction.
\end{proof}
\begin{remark}
Statements 1. and 2. of Theorem \ref{thm:GlobalConvvm} require only that $\overline\varphi_k\in\Phiad$ is chosen such that the search directions $v_k = \overline\varphi_k-\varphi_k$ are gradient related descent directions, as can be seen in the proof above. Hence $\overline \varphi_k$ does not have to be $\mathcal P_k(\varphi_k)$ in Algorithm \ref{algorithm1}. In this case assumption \ref{ass:Bdd} is also not required.
\end{remark}

\section{An application in structural topology optimization based on a phase field model}\label{sec:Application}
In this section we give an example of an optimization problem described in \cite{bfgs2013}, which is not differentiable in a Hilbert space, so the classical projected gradient method cannot be applied, but the assumptions for the VMPT method are fulfilled.\\
We consider the problem of distributing $N$ materials, each with different elastic properties and fixed volume fraction, within a design domain $\Omega\subseteq \R^d$, $d\in\N$, such that the mean compliance $\int_{\Gamma_g}\bfg\cdot \bfu$ is minimal under the external force $\bfg$ acting on $\Gamma_g\subseteq\partial\Omega$. The displacement field $\bfu: \Omega\to \R^d$ is given as the solution of the equations of linear elasticity \eqref{equ:Elasticity}.
To obtain a well posed problem a perimeter penalization is typically used. 
Using phase fields in topology optimization was introduced by Bourdin and Chambolle \cite{BourdinChambolle}.
Here, the $N$ materials are described by a vector valued phase field $\bfvarphi:\Omega\to\R^N$ with $\bfvarphi\geq 0$ and $\sum_i \varphi_i = 1$, which is able to handle topological changes implicitly. The $i$th material is characterized by $\{\bfvarphi_i=1\}$ and the different materials are separated by a thin interface, whose thickness is controlled by the phase field parameter $\varepsilon>0$.  In the phase field setting the perimeter is approximated by the Ginzburg Landau energy. In \cite{bghr2014} it is shown that the given problem for $N=2$ converges as $\varepsilon\to 0$ in the sense of $\Gamma$-convergence. For further details about the model we refer the reader to \cite{bfgs2013}. The resulting optimal control problem reads with $E(\bfphi):=\int_\Omega\left\{ \frac{\varepsilon}{2}|\nabla \bfvarphi|^2 + \frac{1}{\varepsilon}\psi_0(\bfvarphi)\right\}$
\begin{align}\label{equ:MultPhaseMCP}
	\min \tilde J(\bfvarphi,\bfu) :=  & \int_{\Gamma_g}\bfg\cdot \bfu + \gamma E(\bfvarphi)\\
	\bfvarphi\in H^1(\Omega)^N,\ &\bfu \in H^1_D := \{H^1(\Omega)^d\mid \bfxi|_{\Gamma_D}=0\} \notag\\
	\text{subject to }\quad \int_\Omega \bfC(\bfvarphi)\E(\bfu):\E(\bfxi) &= \int_{\Gamma_g}\bfg\cdot \bfxi \quad\forall \bfxi\in H^1_D \label{equ:Elasticity}\\
	\strokedint_\Omega \bfvarphi = \m, \quad	\bfvarphi &\geq 0, \quad	\sum_{i=1}^{N}\varphi^i \equiv 1 ,\label{equ:Constraints}
\end{align}
where $\gamma>0$ is a weighting factor, $\strokedint_\Omega \bfvarphi:= \frac{1}{|\Omega|}\int_\Omega \bfvarphi$,  $\psi_0:\R^N\to \R$ is the smooth part of the potential forcing the values of $\bfvarphi$ to the standard basis $\bm e_i\in \R^N$, and $A:B := \sum_{i,j=1}^d A_{ij}B_{ij}$ for $A,B\in\R^{d\times d}$. The materials are fixed on the Dirichlet domain $\Gamma_D\subseteq\partial\Omega$. The tensor valued mapping $\bfC:\R^N\to \R^{d\times d}\otimes (\R^{d\times d})^*$ is a suitable interpolation of the stiffness tensors $\bfC(\bm e_i)$ of the different materials and $\E(\bfu) := \frac 1 2 (\nabla\bfu + \nabla \bfu^T)$ is the linearized strain tensor. The prescribed volume fraction of the $i$th material is given by $\mathfrak m_i$. For examples of the functions $\psi_0$ and $\bfC$ we refer to \cite{bfgrs2014,bfgs2013}. Existence of a minimizer of the problem \eqref{equ:MultPhaseMCP} as well as the unique solvability of the state equation \eqref{equ:Elasticity} is shown in \cite{bfgs2013} under the following assumptions, which we claim also in this paper.
\newcounter{AssCountP}
\renewcommand{\theAssCountP}{\textbf{(AP)}}    
\begin{list}{\theAssCountP}{\usecounter{AssCountP}}\setlength{\itemsep}{0pt}
 \item \label{ass:psiC11}\label{ass:CC11}\label{ass:Csym}\label{ass:Ccoercive} 
$\Omega\subseteq\R^d$ is a bounded Lipschitz domain; $\Gamma_D,\Gamma_g\subseteq\partial\Omega$ with $\Gamma_D\cap \Gamma_g = \emptyset$ and $\mathcal H^{d-1}(\Gamma_D)>0$. Moreover, $\bfg \in L^2(\Gamma_g)^d$ and $\psi_0\in C^{1,1}(\R^N)$ as well as $\m\geq 0$, $\sum_{i=1}^N \mathfrak m_i = 1$. For the stiffness tensor we assume $\bfC = (C_{ijkl})_{i,j,k,l=1}^d$ with $C_{ijkl}\in C^{1,1}(\R^N)$ and $C_{ijkl}=C_{jikl}=C_{klij}$ and that there exist $a_0, a_1, C >0$, s.t. $a_0|\bm A|^2\leq \bfC(\bfvarphi)\bm A:\bm A\leq a_1|\bm A|^2$ as well as $|\bfC'(\bfvarphi)|\leq C$ holds for all symmetric matrices $\bm A\in\R^{d\times d}$ and for all $\bfvarphi\in\R^N$.
\end{list}
The state $\bfu$ can be eliminated using the control-to-state operator $S$, resulting in the reduced cost functional $\tilde j(\bfvarphi) := \tilde J(\bfvarphi,S(\bfvarphi))$. In \cite{bfgs2013} it is also shown that $\tilde j:H^1(\Omega)^N\cap L^\infty(\Omega)^N\to \R$ is everywhere \Frechet{} differentiable with derivative
\begin{align}\label{eq:jderiv}
	\tilde j'(\bfphi)\bm v = \gamma \int_\Omega \{ \varepsilon \nabla \bfphi:\nabla \bm v + \frac{1}{\varepsilon}{ \psi_0'}(\bfphi)\bm v \} - \int_\Omega \bfC'(\bfphi)\bm v\E(\bfu):\E(\bfu)
\end{align}
for all $\bfvarphi, \bm v\in H^1(\Omega)^N\cap L^\infty(\Omega)^N$, where $\bfu=S(\bfvarphi)$ and $S:L^\infty(\Omega)^N\to H^1(\Omega)^d$ is \Frechet{} differentiable. 
By the techniques in \cite{bfgs2013} one can also show that $S'$ is continuous.

In \cite{bfgs2013,bgsssv2010} the problem is solved numerically by a pseudo time stepping method with fixed time step, which results from an $L^2$-gradient flow approach. 
An $H^{-1}$ gradient flow approach is also considered in \cite{bgsssv2010}. The drawbacks of these methods are that no convergence results to a stationary point exist, 
and hence also no appropriate stopping criteria are known. In addition, typically the methods are very slow, i.e. many time steps are needed until the changes in the solution $\bfvarphi$ or in $j$ are small. Here we apply the VMPT method, which does not have these drawbacks and which can additionally incorporate second order information.

Since $H^1(\Omega)^N\cap L^\infty(\Omega)^N$ is not a Hilbert space the classical projected gradient method cannot be applied. In the following we show that problem \eqref{equ:MultPhaseMCP} fulfills the assumptions on the VMPT method. Amongst others we use the inner product $a_k(\bff,\bfg) = \int_\Omega \nabla \bff:\nabla \bfg$. To guarantee positive definiteness of this $a_k$ we first have to translate the problem by a constant to gain $\int_\Omega\bfvarphi = 0$, which allows us to apply a \Poincare{} inequality. Therefor we perform a change of coordinates in the form $\tilde\bfvarphi = \bfvarphi-\m$ and get the following problem for the transformed coordinates.
\begin{align}\label{equ:MultPhaseMCPTrans}
	\min j(\bfvarphi) := \int_{\Gamma_g}\bfg\cdot S(\bfvarphi+\m) + \gamma E(\bfvarphi+\m)\\
	\bfvarphi\in \Phiad := \left\{\bfvarphi\in H^1(\Omega)^N\relmiddle \strokedint_\Omega \bfvarphi = 0, \quad	\bfvarphi \geq -\m, \quad	\sum_{i=1}^{N}\varphi^i \equiv 0\right\}.\notag
\end{align}
On the transformed problem \eqref{equ:MultPhaseMCPTrans} we apply the VMPT method in the spaces
\begin{align*}
	\Hil &:= \left\{\bfvarphi\in H^1(\Omega)^N\relmiddle \strokedint_\Omega \bfvarphi = \bm 0\right\},\quad	\Dual := L^\infty(\Omega)^N.
\end{align*}
The space of mean value free functions $\Hil$ becomes a Hilbert space with the inner product $(\bm f,\bm g)_\Hil:=(\nabla \bm f,\nabla \bm g)_{L^2}$ and $\|.\|_\Hil$ is equivalent to the $H^1$-norm \cite{AltFunct}.

\begin{theorem}\label{thm:C1}
The reduced cost functional $j : \Hil\cap \Dual\to \R$ is continuously \Frechet{} differentiable and $j'$ is Lipschitz continuous on $\Phiad$.
\end{theorem}
\begin{proof}
The \Frechet{} differentiability of $j$ on $\Hil\cap \Dual$ is shown in \cite{bfgs2013}. Let $\bfeta,\bfvarphi_i\in \Hil\cap \Dual$ and $\bfu_i = S(\bfvarphi_i)$, $i=1,2$. Then with \eqref{eq:jderiv}, $\psi_0\in C^{1,1}(\R^N)$, $C_{ijkl}\in C^{1,1}(\R^N)$ and $|\bfC'(\bfvarphi)|\leq C$ $\forall\bfvarphi\in\R^N$ we get
\begin{align}
	|(j'(\bfvarphi_1)- j'(\bfvarphi_2))\bfeta| &\leq \gamma\varepsilon\|\bfvarphi_1-\bfvarphi_2\|_{H^1}\|\bfeta\|_{H^1} + C\frac{\gamma}{\varepsilon}\|\bfvarphi_1-\bfvarphi_2\|_{L^2}\|\bfeta\|_{L^2} \notag\\
	&\phantom{=\ }+ |\smallint_\Omega(\bfC'(\m+\bfvarphi_1)-\bfC'(\m+\bfvarphi_2))(\bfeta)\E(\bfu_1):\E(\bfu_1)|\notag\\
	&\phantom{=\ }+|\smallint_\Omega\bfC'(\m+\bfvarphi_2)(\bfeta)\E(\bfu_1-\bfu_2):\E(\bfu_1)|\notag\\
	&\phantom{=\ } + |\smallint_\Omega\bfC'(\m+\bfvarphi_2)(\bfeta)\E(\bfu_2):\E(\bfu_1-\bfu_2)|\notag\\
	&\leq C\|\bfvarphi_1-\bfvarphi_2\|_{H^1}\|\bfeta\|_{H^1}\notag\\
	&\phantom{=\ } +  \|(\bfC'(\m+\bfvarphi_1)-\bfC'(\m+\bfvarphi_2))\bfeta\|_{L^\infty}\|\bfu_1\|^2_{H^1}+\notag\\
	&\phantom{=\ } + C \|\bfeta\|_{L^\infty}\|\bfu_1-\bfu_2\|_{H^1}(\|\bfu_1\|_{H^1}+\|\bfu_2\|_{H^1})\notag\\
	&\leq C\|\bfeta\|_{H^1\cap L^\infty}\{\|\bfvarphi_1-\bfvarphi_2\|_{H^1} + \|\bfvarphi_1-\bfvarphi_2\|_{L^\infty}\|\bfu_1\|^2_{H^1} \notag\\
	&\phantom{=\ } + \|\bfu_1-\bfu_2\|_{H^1}(\|\bfu_1\|_{H^1}+\|\bfu_2\|_{H^1})\} \label{eq:derivestim}
\end{align}
To show the continuity of $j'$, let $\bfvarphi_n,\bfvarphi\in \Hil\cap\Dual$ for $n\in\N$ with $\bfvarphi_n\to \bfvarphi$ in $\Hil\cap\Dual$. Using \eqref{eq:derivestim} yields
\begin{multline*}
	\|j'(\bfvarphi_n)- j'(\bfvarphi)\|_{(H^1\cap L^\infty)^*} \\
	\leq C(\|\bfvarphi_n-\bfvarphi\|_{H^1\cap L^\infty}(1+\|\bfu_n\|_{H^1}^2)+\|\bfu_n-\bfu\|_{H^1}(\|\bfu_n\|_{H^1}+\|\bfu\|_{H^1})),
\end{multline*}
where $\bfu_n = S(\bfvarphi_n)$ and $\bfu = S(\bfvarphi)$. From the continuity of $S$ we get that $\|\bfu_n\|_{H^1}$ is bounded and that $\|\bfu_n-\bfu\|_{H^1}\to 0$ as $n\to \infty$. This implies
\begin{align*}
	\|j'(\bfvarphi_n)- j'(\bfvarphi)\|_{(H^1\cap L^\infty)^*} \to 0
\end{align*}
and thus $j\in C^1(\Hil\cap\Dual)$.\\
For the Lipschitz continuity of $j'$ we employ estimate \eqref{eq:derivestim} with $\bfvarphi_i\in \Phi_{ad}$, $i=1,2$. Since $\Phi_{ad}$ is bounded in $L^\infty$, we get that $S$ is Lipschitz continuous on $\Phiad$ and that $\|S(\bfvarphi)\|_{H^1}\leq C$, independent of $\bfvarphi\in\Phiad$, see \cite{bfgs2013}. This yields
\begin{align*}
	\|j'(\bfvarphi_1)- j'(\bfvarphi_2)\|_{(H^1\cap L^\infty)^*} \leq C \|\bfvarphi_1-\bfvarphi_2\|_{H^1\cap L^\infty},
\end{align*}
which proofs the Lipschitz continuity of $j'$ in $\Phi_{ad}$.
\end{proof}

\begin{cor}\label{cor:jPhiadassum}
	The spaces $\Hil$ and $\Dual$, together with $j$ and $\Phiad$ given in \eqref{equ:MultPhaseMCPTrans} fulfill the assumptions \ref{ass:LimitsUnique}-\ref{ass:WeakDiff} of the VMPT method.
\end{cor}
\begin{proof}
	Given the choices for $\Hil$ and $\Dual$ \ref{ass:LimitsUnique} is fulfilled. For $\bfvarphi\in\Phiad$ we have
	\begin{align*}
		\bm {-1}\leq-\m\leq \bfvarphi \leq \bm 1-\m\leq \bm 1\quad \forall \bfvarphi\in\Phiad
	\end{align*}
	almost everywhere in $\Omega$. Thus it holds \ref{ass:Bdd} and $\Phiad\subseteq\Hil\cap\Dual$. Moreover, $\bm 0\in \Phiad$, $\Phiad$ is convex, and since $\Phiad$ is closed in $L^2(\Omega)^N$, it is also closed in $\Hil\hookrightarrow L^2(\Omega)^N$. Thus \ref{ass:Closed} holds.\\
	Assumption \ref{ass:BddBelow} is shown in \cite{bfgs2013} and Theorem \ref{thm:C1} provides \ref{ass:Diff}.\\
	Given
	\begin{align*}
		\spr{j'(\bfvarphi),\bfvarphi_i} &= \int_\Omega\{\gamma \varepsilon \nabla\bfvarphi: \nabla\bfvarphi_i 
+ (\tfrac \gamma \varepsilon \nabla \psi_0(\bfvarphi+\m) -\nabla\bfC(\bfvarphi+\m)\E(\bfu):\E(\bfu))\cdot \bfvarphi_i\}
	\end{align*}
the first term converges to $0$ if $\bfvarphi_i\to 0$ weakly in $H^1$. 
With \ref{ass:psiC11} and $\bfu \in  H^1_D$ we have that
$\tfrac \gamma \varepsilon\nabla \psi_0(\bfvarphi+\m)-\nabla\bfC(\bfvarphi+\m)\E(\bfu):\E(\bfu)\in L^1(\Omega)^N$.
Hence the remaining term converges to $0$ if $\bfvarphi_i\to 0$ weakly-* in $L^\infty$,
which proves that \ref{ass:WeakDiff} is fulfilled.
\end{proof}
\medskip
Possible choices of the inner product $a_k$ for the VMPT method are the inner product on $\Hil$, i.e.
\begin{align}
	a_k(\bm p, \bm y) &= (\bm p, \bm y)_\Hil = \int_\Omega\nabla \bm p:\nabla \bm y \label{eq:akH1}
\end{align}
and the scaled version
$
	a_k(\bm p, \bm y) = \gamma\varepsilon(\bm p, \bm y)_\Hil 
$.
Both fulfill the assumptions \ref{ass:aInnerProduct}-\ref{ass:aToZ}.
We also give an example of a pointwise choice of an inner product, which includes second order information. 
Since this choice is not continuous in $\Hil$, it is not obvious that it fulfills the assumptions. To motivate the choice of this inner product we look at the second order derivative of $j$, which is formally given by
\begin{align*}
  j''(\bfvarphi_k)[\bm p,\bm y] &= \int_\Omega\{\gamma\varepsilon \nabla\bm p:\nabla\bm y - 2(\bfC'(\m+\bfvarphi_k)(\bm y)\E(S'(\bfvarphi_k)\bm p):\E(\bfu_k)) + \\
   &\phantom{=\ } +  \frac{\gamma}{\varepsilon}\nabla^2\psi_0(\m+\bfvarphi_k)\bm p\cdot\bm y - \bfC''(\m+\bfvarphi_k)[\bm p,\bm y]\E(\bfu_k):\E(\bfu_k)\}.
\end{align*}
In \cite{bfgs2013} it is shown that $\bfz_p:=S'(\bfvarphi_k)\bm p\in H^1_D$ is the unique weak solution of the linearized state equation 
\begin{align}\label{eq:linearized}
	\int_\Omega \bfC(\m+\bfvarphi_k)\E(\bm z_p):\E(\bfeta) = -\int_\Omega \bfC'(\m+\bfvarphi_k)\bm p\E(\bfu_k):\E(\bfeta)\quad\forall \bfeta\in H^1_D
\end{align} 
and that $\|\bfz_p\|_{H^1}\leq C\|\bm p\|_{L^\infty}$ holds.
Since the first two terms in $j''$ define an inner product (see proof of Theorem \ref{thm:akso}), we use
\begin{align}\label{eq:akso}
	a_k(\bm p, \bm y) &= \gamma\varepsilon (\bm p,\bm y)_\Hil - 2 \int_\Omega \bfC'(\m+\bfphi_k)(\bm y) \E(\bm z_p):\E(\bfu_k)
\end{align}
as an approximation of $j''(\bfvarphi_k)$. 
Testing equation \eqref{eq:linearized} for $\bfz_y=S'(\bfvarphi_k)\bm y$ with $\bfz_p $ we can equivalently write
\begin{align}\label{eq:akso2}
	a_k(\bm p, \bm y)= \gamma\varepsilon (\bm p,\bm y)_\Hil + 2\int_\Omega \bfC(\m+\bfvarphi_k)\E(\bm z_p):\E(\bm z_y).
\end{align}
We would like to mention that the $C^2$-regularity of $j$ is not necessary for this definition of $a_k$.

\begin{theorem}\label{thm:akso}
	The bilinear form $a_k$ given in \eqref{eq:akso} fulfills the assumptions \ref{ass:aInnerProduct}-\ref{ass:aToZ}.
\end{theorem}
\begin{proof}
	Due to \ref{ass:Ccoercive} and \eqref{eq:akso2} we have
	\begin{align*}
		a_k(\bm p,\bm p) \geq \gamma\varepsilon \|\bm p\|^2_{\Hil}.
	\end{align*}
	Thus, \ref{ass:aInnerProduct} and \ref{ass:aCoercive} is fulfilled. Furthermore, \ref{ass:aBdd} holds due to
	\begin{align*}
		a_k(\bfp,\bm y)&\leq \gamma\varepsilon \|\bfp\|_{H^1}\|\bm y\|_{H^1}+C\|\bfz_p\|_{H^1}\|\bfz_y\|_{H^1}\\
		&\leq \gamma\varepsilon \|\bfp\|_{H^1}\|\bm y\|_{H^1}+C\|\bfp\|_{L^\infty}\|\bm y\|_{L^\infty}\leq C\|\bfp\|_{\Hil\cap\Dual}\|\bm y\|_{\Hil\cap\Dual}.
	\end{align*}
	\ref{ass:aWeakDiff} is proved as in Corollary \ref{cor:jPhiadassum}.\\
	Finally we prove \ref{ass:aToZ}. For $\bm y_k\to 0$ and $\bm p_k\to \bm p$ in $\Hil$ we have $(\bm y_k,\bm p_k)_\Hil\to 0$ for $k\to\infty$. 
With $\bfvarphi_k\to\bfvarphi$, $\bm p_k\to \bm p$ in $\Dual=L^\infty(\Omega)^N$ and $S:L^\infty(\Omega)^N\to H^1(\Omega)^N$ continuously \Frechet{} differentiable, we have $\bfu_k=S(\bfvarphi_k) \to S(\bfvarphi)=:\bfu$ in $H^1_D$  and $\bfz_{p_k} = S'(\bfvarphi_k)\bfp_k\to S'(\bfvarphi)\bfp =: \bfz_p$ in $H^1_D$. In particular, the sequences are bounded in the corresponding norms, including $\|\bm y_k\|_{L^\infty}\leq C$ if $\bm y_k\to\bm y$ weakly-* in $L^\infty$. Using the Lipschitz continuity and boundedness of $\bfC'$ and $\nabla\bfC(\m+\bfvarphi)\E(\bfz_{p}):\E(\bfu)\in L^1(\Omega)^N$ we have
	\begin{align*}
		&|\smallint_\Omega\bfC'(\m+\bfvarphi_k)\bm y_k\E(\bfz_{p_k}):\E(\bfu_k)| \\
		&\leq |\smallint_\Omega(\bfC'(\m+\bfvarphi_k)-\bfC'(\m+\bfvarphi))\bm y_k\E(\bfz_{p_k}):\E(\bfu_k)|\\
		&\phantom{\leq\ } + |\smallint_\Omega\bfC'(\m+\bfvarphi)\bm y_k\E(\bfz_{p_k}-\bfz_{p}):\E(\bfu_k)|\\
		&\phantom{\leq\ } + |\smallint_\Omega\bfC'(\m+\bfvarphi)\bm y_k\E(\bfz_{p}):\E(\bfu_k-\bfu)| + |\smallint_\Omega\bfC'(\m+\bfvarphi)\bm y_k\E(\bfz_{p}):\E(\bfu)|\\
		&\leq L\|\bfvarphi_k-\bfvarphi\|_{L^\infty}\|\bm y_k\|_{L^\infty}\|\bfz_{p_k}\|_{H^1}\|\bfu_k\|_{H^1} \\
		&\phantom{\leq\ } +\|\bfC'(\m+\bfvarphi)\|_{L^\infty}\|\bm y_k\|_{L^\infty}\|\bfz_{p_k}-\bfz_{p}\|_{H^1}\|\bfu_k\|_{H^1}\\
		&\phantom{\leq\ } + \|\bfC'(\m+\bfvarphi)\|_{L^\infty}\|\bm y_k\|_{L^\infty}\|\bfz_{p}\|_{H^1}\|\bfu_k-\bfu\|_{H^1} \\
		&\phantom{\leq\ } + |\smallint_\Omega (\nabla\bfC(\m+\bfvarphi)\E(\bfz_{p}):\E(\bfu))\cdot \bm y_k| \to 0,
	\end{align*}
	which gives \ref{ass:aToZ}.
\end{proof}

Hence with $0<\lambda_{min}\leq \lambda_k\leq \lambda_{max}$, all assumptions of Theorem \ref{thm:GlobalConvvm} are fulfilled and we get global convergence in the space $H^1(\Omega)^N\cap L^\infty(\Omega)^N$.

\section{Numerical results}
We discretize the structural topology optimization problem \eqref{equ:MultPhaseMCP}-\eqref{equ:Constraints} using standard piecewise linear finite elements
for  the control $\bfvarphi$ and the state variable $\bfu$. The projection type subproblem \eqref{eq:projproblvm} is solved by a primal dual active set (PDAS) method similar to the method described in \cite{bgss2011}.
Many numerical examples for this problem can be found in \cite{bfgrs2014, bghr2014}, e.g. for cantilever beams with up to three materials in two or three space dimensions and for an optimal material distribution within an airfoil. In \cite{bfgrs2014} the choice of the potential $\psi$ as an obstacle potential and the choice of the tensor interpolation $\bfC$ is discussed. Also the inner products $(.,.)_\Hil$ and $\gamma\varepsilon(.,.)_\Hil$ for fixed scaling parameter $\lambda_k=1$ are compared, where both give rise to a mesh independent method and the latter leads to a large speed up. Note that the choice of $(.,.)_\Hil$ with $\lambda_k = (\gamma\varepsilon)^{-1}$ leads to the same iterates than choosing $\gamma\varepsilon(.,.)_\Hil$ and $\lambda_k=1$. Furthermore, it is discussed in \cite{bfgrs2014} that the choice of $\gamma\varepsilon(.,.)_\Hil$ can be motivated using $j''(\bfphi)$ or by the fact that for the minimizers $\{\bfphi_\varepsilon\}_{\varepsilon>0}$ the Ginzburg-Landau energy converges to the perimeter as $\varepsilon\to 0$ and hence $\gamma\varepsilon\|\bfphi_\varepsilon\|^2_\Hil\approx const$ independent of $\varepsilon\ll 1$. However, since this holds only for the iterates $\bfphi_k$ when the phases are separated and the interfaces are present with thickness proportional to $\varepsilon$, we suggest to adopt $\lambda_k$ in accordance to this. As updating strategy for $\lambda_k$ the following method is applied: 
Start with $\lambda_0 = 0.005(\gamma\varepsilon)^{-1}$, then if $\alpha_{k-1}=1$ set $\tilde\lambda_k = \lambda_{k-1}/0.75$, else $\tilde\lambda_k = 0.75\lambda_{k-1}$ and $\lambda_k = \max\{\lambda_{min},\min\{\lambda_{max},\tilde \lambda_k\}\}$. The last adjustment yields that \ref{ass:Lambda} is fulfilled. Numerical experiments in \cite{bfgrs2014} show that this in fact produces for the choice $(.,.)_\Hil$ a scaling with $\lambda_k\approx (\gamma\varepsilon)^{-1}$ for large $k$.\\
In \cite{bfgrs2014, bghr2014} the effect of obtaining various local minima of the nonconvex optimization problem \eqref{equ:MultPhaseMCP}-\eqref{equ:Constraints} 
by choosing different initial guesses $\bfphi_0$ can be seen. However also the other parameters have an influence.\\
In this paper we concentrate on comparing different choices of the inner products $a_k$ and use herefor the cantilever beam described in \cite{bfgrs2014} with $\psi_0(\bfvarphi) = \tfrac 1 2 (1-\bfphi\cdot\bfphi)$ and a quadratic interpolation of the stiffness tensors $\bfC(\bfvarphi)$. The computation are performed on a personal computer with 3GHz and 4GB RAM. First we discuss the choice of $(.,.)_{L^2}$ versus $(.,.)_\Hil$. The choice of the $L^2$-inner product leads to the commonly used projected $L^2$-gradient method. However, $(.,.)_{L^2}$ does not fulfill the assumptions of the VMPT method, since $j$ is not differentiable in $L^2(\Omega)^N$ or $L^2(\Omega)^N\cap L^\infty(\Omega)^N$. Thus, global convergence is given for the discretized, finite dimensional problem but not in the continuous setting. This leads in contrast to the choice of $(.,.)_\Hil$ to mesh dependent iteration numbers for the $L^2$-gradient method, which can be seen in Table \ref{tab:L2H1}. The values in Table \ref{tab:L2H1} were computed for different uniform mesh sizes $h$ with the parameters $\varepsilon=0.04$, $\gamma=0.5$, $\bfphi_0 \equiv \m$ and $tol=10^{-5}$ for the stopping criterion $\sqrt{\gamma\varepsilon}\|\nabla\bfphi_k\|_{L^2}\leq tol$. The behaviour of iteration numbers is in accordance to our analytical results in function spaces considering $h\to 0$. Furthermore, numerical results not listed here show that we obtain 
for $(.,.)_\Hil$ and large $k$ scalings $\lambda_k \approx (\gamma\varepsilon)^{-1}$  independent of the mesh parameter $h$, whereas the $L^2$-inner product produces $\lambda_k$ scaled with $h^2$. Since the algorithm using the $L^2$-inner product is equivalent to the explicit time discretization of the $L^2$-gradient flow, i.e. of the Allen-Cahn variational inequality coupled with elasticity, with time step size $\Delta t = \lambda_k$, the scaling $\lambda_k = \grossO(h^2)$ reflects the known stability condition $\Delta t = \grossO(h^2)$ for explicit time discretizations of parabolic equations.\\

\begin{table}[h]\centering
\begin{tabular}{|l||r|r|r|r|r|}
  \hline
 $ h$ & $2^{-4}$ &  $2^{-5}$ & $2^{-6}$ & $2^{-7}$ & $2^{-8}$\\
  \hline\hline
$(.,.)_{L^2}$ & 323   &  5015  & 18200  & 57630  & 172621\\
$(.,.)_\Hil$ & 111 & 407 & 320 & 275 & 269\\
  \hline
\end{tabular}
\caption{Comparison of iteration numbers for $(.,.)_{L^2}$ and $(.,.)_\Hil$.}\label{tab:L2H1}
\end{table}

Next we compare $(.,.)_\Hil$ with $a_k$ given in \eqref{eq:akso}, which incorporates second order information. As experiment we again use the cantilever beam in \cite{bfgrs2014}, now with $\varepsilon = 0.001$, $\gamma = 0.002$, $tol = 10^{-4}$ and random initial guess $\bfvarphi_0$ together with an adaptive mesh, which is fine on the interface with $h_{max} = 2^{-6}$ and $h_{min} = 2^{-11}$. The parameter $\lambda_k$ is updated as described above. The computational costs of one iteration with $a_k$ given in \eqref{eq:akso} is significantly higher, since the calculation of $\mathcal P_k(\bfphi_k)$ requires the solution of a quadratic optimization problem with $\bfphi\in\Phiad$ and in addition with the linearized state equation \eqref{eq:linearized} as constraints. However, in each PDAS iteration solving the subproblem for fixed $k$, only the right hand side of \eqref{eq:linearized} changes, namely only $\bm p$. We factorize the matrix in the discrete equation once such that for each $\bm p$ only a cheap forward and backward substitution has to be done. In Table \ref{tab:akCompare} the corresponding iteration numbers, the total CPU time, the values of the combined cost functional $j(\bfphi^*)$ as well as of the parts, i.e. the mean compliance and the Ginzburg-Landau energy are listed. One observes the drastic reduction in iteration numbers using second order information. Due to the mentioned higher costs of calculating the search directions the total CPU-time is only halved. Nevertheless, this can be possibly improved using a more sophisticated solver for $\mathcal P_k(\bfphi_k)$. It can be also observed that the cost $j(\bfphi^*)$ and the probably more interesting value of the mean compliance is lower. Hence, the different inner products result in different local minima, which are shown in Figure \ref{fig:SmallGamma}. The inner product given in \eqref{eq:akso} yields a finer structure. Also in other experiments we observed a local minima with lower cost value for this choice of $a_k$.

\begin{figure}
\subfloat[$(.,.)_\Hil$.]{\includegraphics[width=6cm]{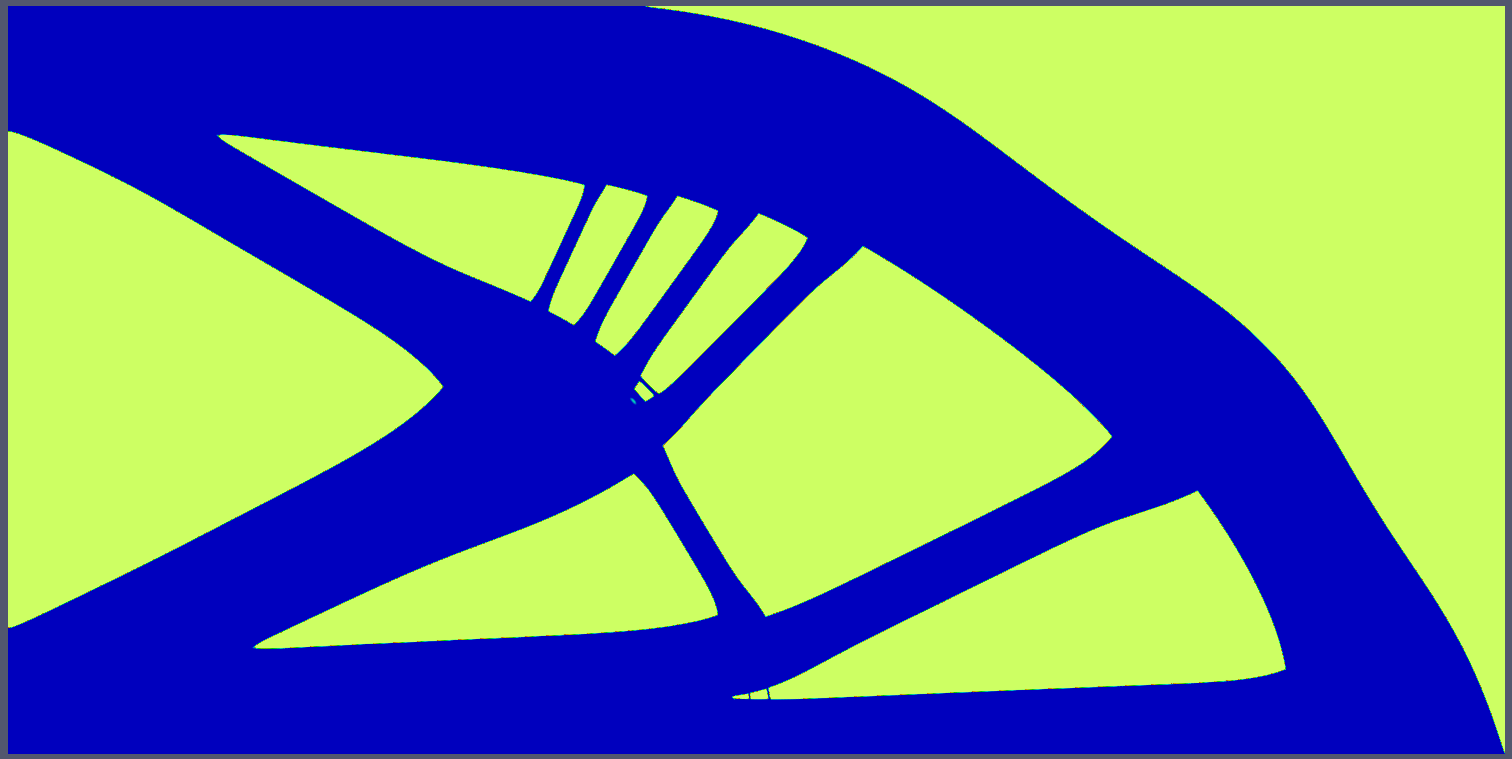}}\phantom{kkkl}\subfloat[$a_k$ given in \eqref{eq:akso}.]{\includegraphics[width=6cm]{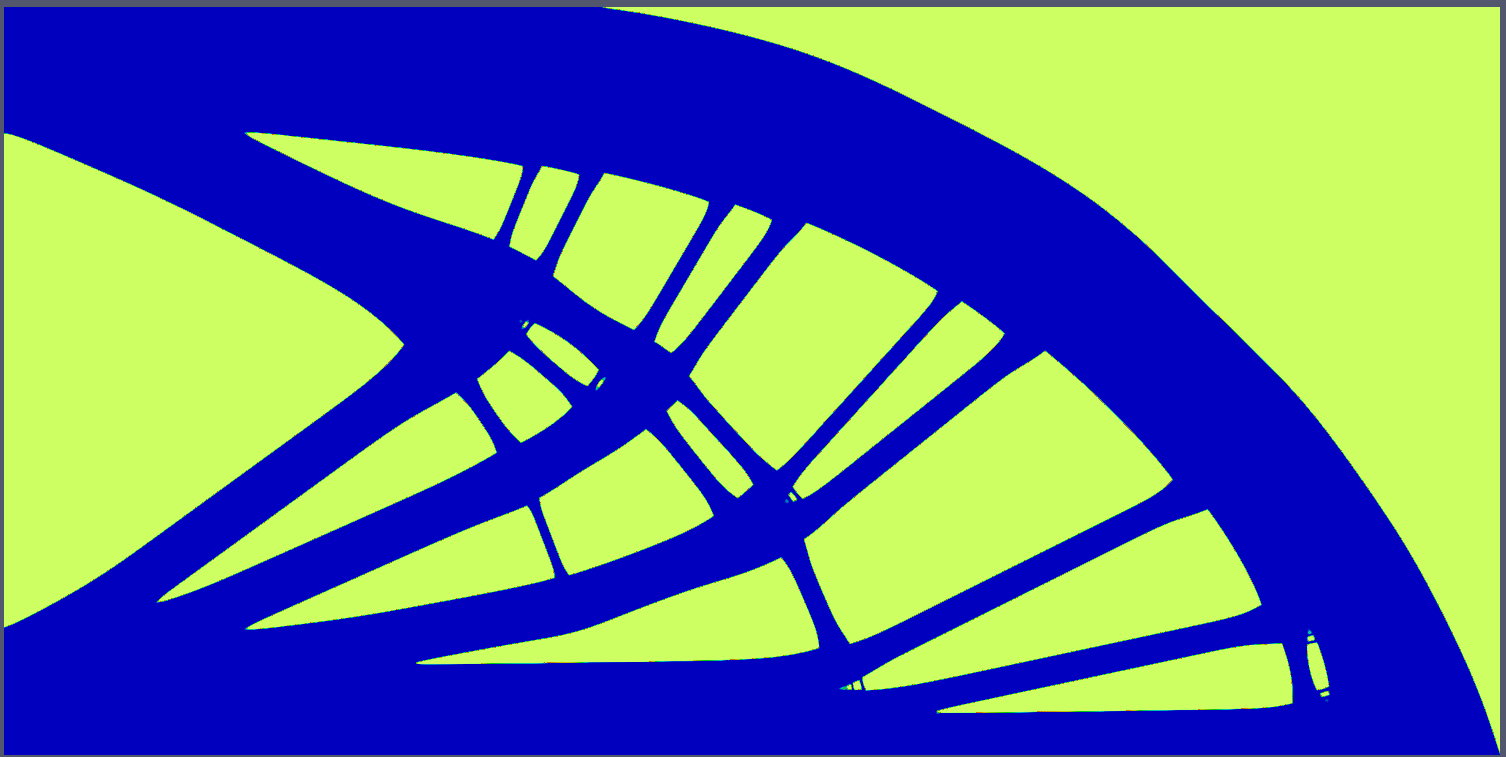}}\\
\caption{Local minima for the cantilever beam.}\label{fig:SmallGamma}
\end{figure}

\begin{table}[h]\centering
\begin{tabular}{|l||r|r|r|r|r|}
  \hline
  inner product & iterations & CPU time & $j(\bfvarphi^*)$ & $\int_{\Gamma_g}\bfg\cdot\bfu^*$ & $E(\bfvarphi)$ \\
  \hline\hline
  $(.,.)_\Hil$ & 11189 & 42h 12min & 15.07 & 15.03 & 20.79\\
  $a_k$ in \eqref{eq:akso} & 851 & 19h & 14.99 & 14.93 & 30.12\\
  \hline
\end{tabular}
\caption{Comparison of two different inner products.}\label{tab:akCompare}
\end{table}


We successfully applied also an L-BFGS update in function spaces (see e.g. \cite{gruver1981algorithmic} for the unconstrained case in Hilbert space) of the metric $a_k$, i.e. starting with $a_0(\bm u,\bm v) = \gamma\varepsilon(\bm u,\bm v)_\Hil$ we use the update 
\\ \centerline{$
a_{k+1}(\bm u,\bm v) = a_k(\bm u,\bm v) - \frac {a_k(\bmsub p_k,\bmsub u) a_k(\bmsub p_k,\bmsub v)}{a_k(\bmsub p_k,\bmsub p_k)} 
+ \frac{\spr{y_k,\bmsub u},\spr{y_k,\bmsub v}}{\spr{y_k,\bmsub p_k}}
$ }\\
in case that $\spr{y_k,\bm p_k}>0$, where $\bm p_k := \bfphi_{k+1}-\bfphi_k$ and	$y_k := j'(\bfphi_{k+1})-j'(\bfphi_k)$, which performs very good especially for small $\gamma$. Note that -- as in the finite dimensional case -- assumption \ref{ass:aCoercive} cannot be shown for this sequence of inner products, but numerical experiments show that the discretized method is mesh independent, see Table \ref{tab:BFGS} for the above cantilever beam example, where the maximal recursion depth is set to 10.
\begin{table}[h]\centering
\begin{tabular}{|r||r|r|r|}
  \hline
  $h$ &  $2^{-5}$ & $2^{-6}$ & $2^{-7}$ \\
  \hline\hline
$H^1$-BFGS iterations & 85   &  88  & 86\\
  \hline
\end{tabular}
\caption{Mesh independent iteration numbers for the $H^1$-BFGS method.}\label{tab:BFGS}
\end{table}

The following compliant mechanism problem
\begin{align*}
\min\; \;   &\frac 1 2\int_{\Omega_{obs}} (1-\varphi^N)|\bfu-\bfu_\Omega|^2 +   \gamma E(\bfphi),
\end{align*}
where the elasticity equation \eqref{equ:Elasticity} and the constraints \eqref{equ:Constraints} have to hold, is more difficult. In our numerical analysis the solution process is more sensitive to the choice of $a_k$. Here the above $H^1$-BFGS approach enables us to solve the problem in an acceptable time. Until $\gamma\varepsilon\|\nabla v_k\|_{L^2}\leq tol=10^{-4}$ the calculation of the material distribution in Figure \ref{fig:Cruncher} took 22 hours. It aims to crunch a nut in the middle of the left boundary when the force acts on the right hand side from above and below and the mechanism is supplied on the left boundary.

Moreover, we also successfully applied the VMPT method on the following drag minimization problem of the Stokes flow using a phase field approach, which is 
analysed in \cite{HechtStokesEnergy}:
\begin{align*}
	\min  \int_\Omega \frac{1}{2}|\nabla \bfu|^2 + &\frac 1 2 \alpha_\varepsilon(\varphi)|\mathbf{u}|^2 + \gamma E(\varphi)\\
	\int_\Omega \alpha_\varepsilon(\varphi)\bfu \mathbf{v} + \int_\Omega \nabla\bfu\cdot\nabla \mathbf{v} &= \mathbf{0} \quad\forall \mathbf{v}\in H^1_{0,div}(\Omega)\\
	\mathbf{u}|_{\partial\Omega} \equiv \left( 1,  0 \right)^T,\quad	\strokedint \varphi &=  0.75,\quad	-1\leq \varphi \leq 1.
\end{align*}
We applied a nested approach in $h$ and $\varepsilon$ as well as an adaptive grid. As inner products we used the above $H^1$-BFGS method and obtained the result in Figure 
\ref{fig:Stokes} with 188 iterations to obtain $tol=10^{-3}$, which took 17 minutes.

A different type of optimization problem is the inverse problem for a discontinuous diffusion coefficient, where the discontinuous coefficient $a$ is smoothed by a phase field approach and no mass conservation is used \cite{DeEllSty2015}:
\begin{align*}
& \qquad \min\; \;   \frac 1 2\int_{\Omega} |u-u_{obs}|^2 + \gamma E(\varphi) \\
	\text{s.t.}\quad& 
\int_\Omega a(\varphi)\nabla u\cdot \nabla\xi = \int_{\Gamma} g \xi
\quad\forall \xi\in H^1\quad \text{and}\quad \int_\Omega u = \int_\Omega u_{obs},\quad
-1 \leq \varphi \leq 1 .
\end{align*}
We choose $u_{obs}$ as solution of the state equation for $\varphi$ shown in the upper part of Figure \ref{fig:InvProbl} with added noise of 5\% and obtain the solution shown in the lower part of Figure \ref{fig:InvProbl}.
\begin{figure}
\subfloat[Crunching mechanism.]{\label{fig:Cruncher}\includegraphics[width=0.3\textwidth,height=0.25\textwidth]{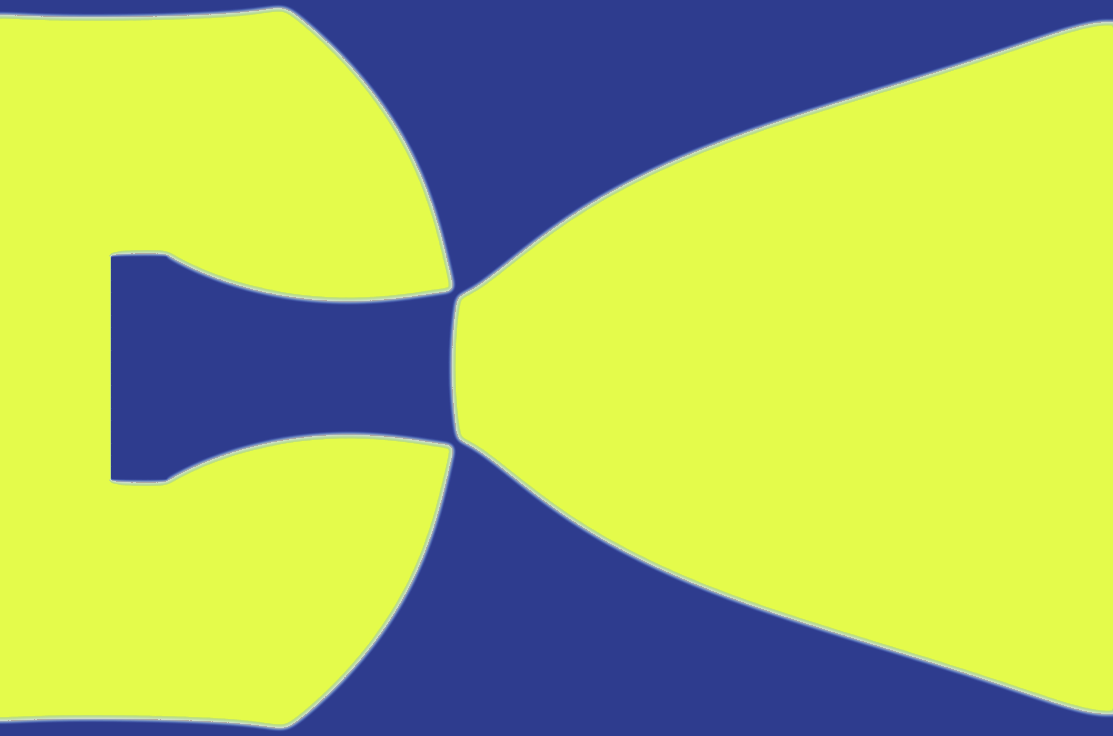}\hspace{0.005\textwidth}}
\subfloat[Obstacle minimizing drag.]{\label{fig:Stokes}\hspace{0.02\textwidth}
\includegraphics[width=0.3\textwidth,height=0.25\textwidth]{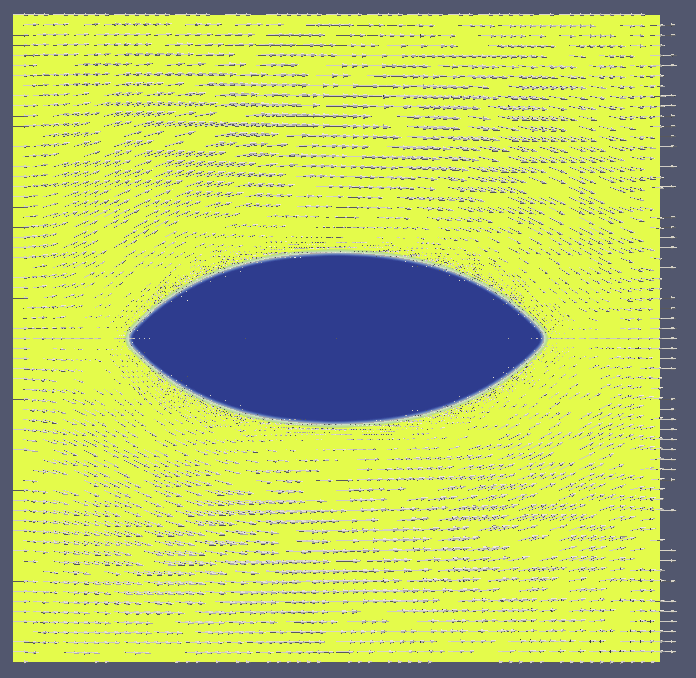}}
\subfloat[Identified coefficient.]{\label{fig:InvProbl}\hspace{0.02\textwidth}
\includegraphics[width=0.25\textwidth,height=0.25\textwidth]{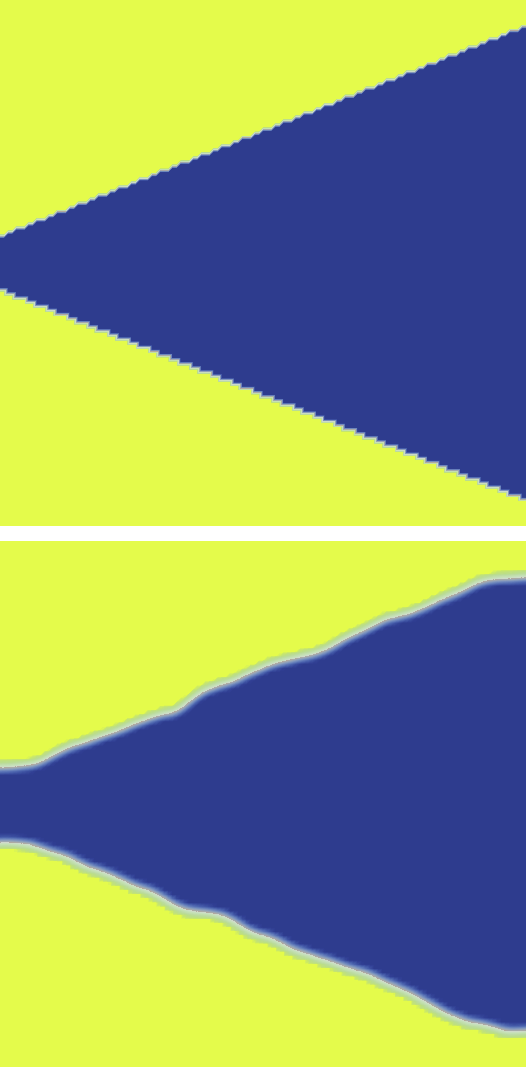}\hspace{0.05\textwidth}}\\
\caption{Successful applications of the VMPT method.}
\end{figure}

The VMPT method can also be used for image inpainting using a phase field approach by considering
\begin{align*}
\min\; \;   &\tfrac 1 2\|\varphi-f\|^2_{H(\Omega\setminus D)} + \gamma E(\varphi)
\end{align*}
such that $\varphi$ fulfills \eqref{equ:Constraints}, where $f$ is the given image and the inpainting is performed in $D$ 
\cite{BuHeSch09}. The method can adjust to the chosen metric $H(\Omega\setminus D)$ and for this problem a line search with exact step length can be applied \cite{KiesMasterThesis}.\\

The last four mentioned application examples are preliminary results and are under further studies. 
To our knowledge the VMPT-method outperforms the existing applied optimization algorithms in these cases.

\bibliographystyle{plain}
\bibliography{literature}


\end{document}